\documentclass[11pt,a4paper]{amsart}
\usepackage{amssymb,amsmath,amsthm}

\usepackage[pdftex]{graphicx}

\usepackage[colorlinks=true, pdfstartview=FitV, linkcolor=blue, citecolor=blue, urlcolor=blue,pagebackref=false]{hyperref}

\usepackage{esint}
\usepackage{mathrsfs}

\usepackage{mathtools}

\usepackage{times,latexsym,color}

\newcommand{\BOX}{\ensuremath\Box}

\newtheorem{theorem}{Theorem}
\newtheorem*{theorem*}{Theorem}
\newtheorem{pro}{Proposition}
\newtheorem{lemma}{Lemma}
\newtheorem{cor}{Corollary}

\theoremstyle{remark}
\newtheorem{remark}{Remark}

\theoremstyle{definition}

\def\XXint#1#2#3{{\setbox0=\hbox{$#1{#2#3}{\int}$ }
\vcenter{\hbox{$#2#3$ }}\kern-.6\wd0}}

\newcommand{\R}{\mathbb{R}}

\newcommand{\ep}{\varepsilon}

\definecolor{darkgreen}{rgb}{0,0.5,0}
\definecolor{darkblue}{rgb}{0,0,0.7}
\definecolor{darkred}{rgb}{0.9,0.1,0.1}
\definecolor{lightblue}{rgb}{0,0.51,1}

\begin{document}

\title{{Higher integrability and the number of singular points for the Navier-Stokes equations with a scale-invariant bound }}

\author[T. Barker]{Tobias Barker}
\address[T. Barker]{Department of Mathematical Sciences, University of Bath, Bath BA2 7AY. UK}
\email{tobiasbarker5@gmail.com}

\subjclass[2010]{}
\date{\today}

\maketitle

\noindent {\bf Abstract} 
First, we show that if the pressure $p$ (associated to a weak Leray-Hopf solution $v$ of the Navier-Stokes equations) satisfies $\|p\|_{L^{\infty}_{t}(0,T^*; L^{\frac{3}{2},\infty}(\mathbb{R}^3))}\leq M^2$, then $v$ possesses higher integrability up to the first potential blow-up time $T^*$. Our method is concise and is based upon energy estimates applied to powers of $|v|$ and the utilization of a `small exponent'.

As a consequence, we show that if a weak Leray-Hopf solution $v$ first blows up at $T^*$ and satisfies the Type I condition $\|v\|_{L^{\infty}_{t}(0,T^*; L^{3,\infty}(\mathbb{R}^3))}\leq M$, then $$\nabla v\in L^{2+O(\frac{1}{M})}(\mathbb{R}^3\times (\tfrac{1}{2}T^*,T^*)).$$
This is the first result of its kind, improving the integrability exponent of $\nabla v$ under the Type I assumption in the three-dimensional setting.

Finally, we show that if $v:\mathbb{R}^3\times [-1,0]\rightarrow \mathbb{R}^3$ is a weak Leray-Hopf solution to the Navier-Stokes equations with $s_{n}\uparrow 0$ such that
$$\sup_{n}\|v(\cdot,s_{n})\|_{L^{3,\infty}(\mathbb{R}^3)}\leq M $$  
then $v$ possesses at most $O(M^{20})$ singular points at $t=0$. Our method is  direct and concise. It is based upon known $\varepsilon$-regularity, global bounds on a Navier-Stokes solution with initial data in $L^{3,\infty}(\mathbb{R}^3)$ and rescaling arguments. We do not require arguments based on backward uniqueness nor unique continuation results for parabolic operators.

\section{Introduction}
In this paper, we consider the three-dimensional incompressible Navier-Stokes equations
\begin{equation}\label{NSintro}
\partial_{t}v-\Delta v+v\cdot\nabla v+\nabla p=0,\,\,\,\nabla\cdot v=0,\,\,\,v(\cdot,0)=v_{0}(x)\,\,\,\textrm{in}\,\,\,\mathbb{R}^3\times (0,T).
\end{equation}
Here, $T\in (0,\infty]$. It is well known that these equations are invariant with respect to the rescaling
\begin{equation}\label{eqrescaling}
(v_{\lambda}(x,t), p_{\lambda}(x,t), v_{0\lambda}(x))=(\lambda v(\lambda x,\lambda^2 t),\lambda^2 p(\lambda x,\lambda^2 t), \lambda v_{0}(\lambda x) )\,\,\textrm{with}\,\,\,\lambda\in (0,\infty).
\end{equation}
In the seminal paper \cite{Le}, Leray showed that for any square-integrable solenodial\footnote{In the sense of distributions.} initial data $v_{0}(x)$ there exists at least one associated global-in-time \textit{weak Leray-Hopf solution} $v$. Such a solution solves \eqref{NSintro} in the distributional sense and satisfies the following energy inequality for all $t\geq 0$:
\begin{equation}\label{energyinequality}
\|v(\cdot,t)\|_{L^{2}(\mathbb{R}^3)}^2+2\int\limits_{0}^{t}\int\limits_{\mathbb{R}^3} |\nabla v|^2 dxds\leq \|v_{0}\|_{L^{2}(\mathbb{R}^3)}^2.
\end{equation}
It remains a long-standing open problem as to whether or not weak Leray-Hopf solutions, with sufficiently smooth initial data, remain smooth for all time.

Since Leray's paper \cite{Le}, a substantial number of authors have investigated conditional regularity criteria for the Navier-Stokes equations \eqref{NSintro}. The vast majority of such results assume a finite \textit{scale-invariant quantity} $F(v,p)$, which is invariant with respect to the Navier-Stokes rescaling \eqref{eqrescaling}. More specifically, $F(v,p)$ is a scale-invariant quantity if
\begin{equation}\label{scaleinvariance}
F(v_{\lambda}, p_{\lambda})=F(v,p)\,\,\,\forall\lambda\in (0,\infty).
\end{equation}
In particular, regularity has been established assuming scale-invariant bounds such as $v\in L^{5}_{x,t}$ \cite{lady} and $v\in L^{\infty}_{t}L^{3}_{x}$ \cite{ESS}. Whilst it is impossible to exhaustively list such results, we also briefly mention the recent results \cite{sereginL3limit}, \cite{albritton} and \cite{albrittonbarker}. 

Importantly many very natural blow-up ans\"{a}tze for the Navier-Stokes equations involve solutions with finite scale-invariant quantities. One such example is the class of \textit{backward self-similar solutions} suggested as a blow-up ansatze by Leray in \cite{Le}. Backward self-similar solutions are invariant with respect to the Navier-Stokes rescaling \eqref{eqrescaling} and take the form
$$v(x,t)=\frac{1}{\sqrt{a(T^*-t)}} V\Big(\frac{x}{\sqrt{a(T^*-t)}}\Big)\,\,\,\textrm{for}\,\,\,(x,t)\in\mathbb{R}^3\times (0,T^*). $$ Here, $a$ is a constant.
Backward self-similar solutions were ruled out in two important papers \cite{necasselfsimilar} and \cite{tsai}.

Another very natural (but more general) blow-up ansatze is that of \textit{backward discretely self-similar solutions}, whose existence remains a long-standing open problem. Here we say that $v:\mathbb{R}^3\times (-\infty,0)\rightarrow\mathbb{R}^3$ is a backward discretely self-similar solution of the Navier-Stokes equations if there exists $\lambda\in (1,\infty)$ such that
$$v(x,t)=\lambda v(\lambda x,\lambda^2 t)\,\,\forall(x,t)\in \mathbb{R}^3\times (-\infty,0). $$
In \cite{chaewolfDSS}, Chae and Wolf showed that solutions satisfying the backward discretely self-similar ansatze (with appropriate decay) must satisfy the scale-invariant bound
\begin{equation}\label{DBSSpointwise}
|v(x,t)|\leq \frac{M}{|x|+\sqrt{-t}}\,\,\,\forall(x,t)\in \mathbb{R}^3\times (-\infty,0]\setminus\{(0,0)\}.
\end{equation}
This implies that
\begin{equation}\label{vweakL3}
\|v\|_{L^{\infty}(-\infty,0; L^{3,\infty}(\mathbb{R}^3))}\leq M,
\end{equation}
where $L^{3,\infty}(\mathbb{R}^3)$ is defined in `2. Preliminaries'.
 The space $L^{3,\infty}(\mathbb{R}^3)$ is slightly larger than $L^{3}(\mathbb{R}^3)$ and contains the function $|x|^{-1}$.

In order to understand whether singularity formation can occur for the Navier-Stokes equations satisfying \eqref{vweakL3} (or more general scale-invariant bounds) it is very natural to first investigate the following questions.
\begin{itemize}
\item []\textrm{\textbf{(Q.1)}} Do singular weak Leray-Hopf solutions of the Navier-Stokes equations satisfying \eqref{vweakL3} possess additional properties other than belonging to $L^{\infty}_{t}L^{3,\infty}_{x}\cap L^{\infty}_{t}L^{2}_{x}\cap L^{2}_{t}\dot{H}^{1}_{x}$ or the corresponding spaces arising from maximal regularity\footnote{By this, we mean estimates coming from the linear theory of the unsteady Stokes equations with forcing
$f=v\cdot\nabla v,\,\,\,\textrm{div}\,v=0\,\,\,\textrm{and}\,\,\,v\in L^{\infty}_{t}L^{3,\infty}_{x}\cap L^{\infty}_{t}L^{2}_{x}\cap L^{2}_{t}\dot{H}^{1}_{x}.$}? 
\medskip
\item []\textrm{\textbf{(Q.2)}} For a solution $v$ with a scale-invariant bound, which first loses smoothness at $T^*$, what is the structure of the set of singular points\footnote{We say that ($x_0,T^*)$ is a \textit{singular point} of $v$ if $v\notin L^{\infty}_{x,t}(B(x_0,r)\times (T^*-r^2, T^*))$ for all sufficiently small $r>0$.} at $T^*$?
\end{itemize} 
In this paper we address the questions \textbf{(Q.1)-(Q.2)}.
\subsection{Statement of Results}
Our first Theorem addresses \textbf{(Q.1)} under a (more general) scale-invariant assumption. In particular, we obtain higher integrability up to a potential first blow-up time $T^*$. Here is our first result.
\begin{theorem}\label{higherintegpres}
There exists a universal constant $C^{(0)}_{univ}\in [1,\infty)$ such that the following holds.

 Let $v$ be a weak Leray-Hopf solution to the Navier-Stokes equations on $\mathbb{R}^3\times (0,\infty)$ and let $p$ be the pressure associated to the solution $v$. Assume that  $v$ first blows up $T^*$, namely 
 \begin{equation}\label{vbounded}
 v\in L^{\infty}_{loc}([0,T^*); L^{\infty}(\mathbb{R}^3))\,\,\,\textrm{and}\,\,\,\lim_{t\uparrow T^*}\|v(\cdot,t)\|_{L^{\infty}(\mathbb{R}^3)}=\infty.\end{equation}
 Assume there exists an $M$ sufficiently large\footnote{Throughout this paper, `$M$ sufficiently large' means $M\geq M_{1}$ for a suitably chosen large universal constant $M_{1}$. } such that
 \begin{equation}\label{pressuretypeI}
 \|p\|_{L^{\infty}((0,T^*); L^{\frac{3}{2},\infty}(\mathbb{R}^3))}\leq M^2. 
 \end{equation}
Then the above assumptions imply that we have higher integrability \textbf{up to $T^*$}. Namely, for all $t_{1}\in (0,T^*)$ we have
\begin{equation}\label{higherintegrable}
\||v|^{\frac{q}{2}}\|_{L^{\infty}((t_1,T^*); L^{2}(\mathbb{R}^3))}+\int\limits_{t_1}^{T^*}\int\limits_{\mathbb{R}^3} |\nabla (|v|^{\frac{q}{2}})|^2 dxds<\infty\,\,\,\textrm{with}\,\,\,q:=2+\frac{C^{(0)}_{univ}}{M}.
\end{equation}
\end{theorem}
As a corollary, we show that if $v$ satisfies $\|v\|_{L^{\infty}(0,T^*; L^{3,\infty}(\mathbb{R}^3))}\leq M$ and first blows up at $T^*$, then $\nabla v\in L^{2+O(\frac{1}{M})}(\mathbb{R}^3\times (\tfrac{1}{2}T^*,T^*)).$ To the best of our knowledge, this is the first result giving improved integrability of $\nabla v$ under such a scale-invariant assumption.
\begin{cor}\label{higherintegtypeI}
There exists a universal constant $C^{(1)}_{univ}\in (0,\infty)$ such that the following holds.

 Let $v$ be a weak Leray-Hopf solution to the Navier-Stokes equations on $\mathbb{R}^3\times (0,\infty)$ and let $p$ be the pressure associated to the solution $v$. Assume that $v$ first blows up at $T^*$, namely
 \begin{equation}\label{vboundedhausdorff}
 v\in L^{\infty}_{loc}([0,T^*); L^{\infty}(\mathbb{R}^3))\,\,\,\textrm{and}\,\,\,\lim_{t\uparrow T^*}\|v(\cdot,t)\|_{L^{\infty}(\mathbb{R}^3)}=\infty.\end{equation}
Assume there exists an $M$ sufficiently large such that
 \begin{equation}\label{typeIhigherinteg}
 \|v\|_{L^{\infty}((0,T^*); L^{3,\infty}(\mathbb{R}^3))}\leq M. 
 \end{equation}
Then the above assumptions imply that we have higher integrability of $\nabla v$ \textbf{up to $T^*$}. Namely, for all $t_{2}\in (0,T^*)$ we have
\begin{equation}\label{higherintegrablegrad v}
\nabla v\in L^{2+\frac{C^{(1)}_{univ}}{M}} (\mathbb{R}^3\times (t_2,T^*)).
\end{equation} 

\end{cor}

Our final Theorem addresses \textbf{(Q.2)}. Namely, it concerns an upper bound for the number of singular points for a solution of the Navier-Stokes equations $v:\mathbb{R}^3\times (-1,0)\rightarrow\mathbb{R}^3$, which is uniformly bounded in $L^{3,\infty}(\mathbb{R}^3)$ along timeslices $s_{n}\uparrow 0$.
\begin{theorem}\label{finiteTypeItimeslicesing}
There exists a universal constant $C^{(2)}_{univ}\in [1,\infty)$ such that the following holds. 

Let $v$ be a weak Leray-Hopf solution to the Navier-Stokes equations on $\mathbb{R}^3\times (-1,\infty)$. Assume that $v$ first blows-up at $0$, namely
$$v\in L^{\infty}_{loc}([-1,0); L^{\infty}(\mathbb{R}^3))\,\,\,\textrm{and}\,\,\,\lim_{t\uparrow 0}\|v(\cdot,t)\|_{L^{\infty}(\mathbb{R}^3)}=\infty. $$
Let $M$ be sufficiently large and assume that there exists an increasing sequence $s^{(n)}\in (-1,0)$ with $s^{(n)}\uparrow 0$ such that $v$ satisfies
\begin{equation}\label{typeItimeslice}
\sup_{n}\|v(\cdot, s^{(n)})\|_{L^{3,\infty}(\mathbb{R}^3)}= M<\infty.
\end{equation}
Let

\begin{equation}\label{sigmadef}
\sigma:=\{x: (x,0)\,\,\,\textrm{is}\,\,\,\textrm{a}\,\,\,\textrm{singular}\,\,\,\textrm{point}\,\,\,\textrm{of}\,\,\,v\}.
\end{equation}
Then the above assumptions imply that $\sigma$ has at most $C^{(2)}_{univ} M^{20}$ elements.
\end{theorem}
\subsection{Comparison with Previous Literature and Novelty of our Results}
\subsubsection{Theorem \ref{higherintegpres} and Corollary \ref{higherintegtypeI}}
Recently in \cite{burczakseregin}, the authors consider the linear Stokes system with drift. In particular,
\begin{align}\label{stokesdrift}
\begin{split}
&\partial_{t}v+b\cdot\nabla v-\Delta v+\nabla p=0\,\,\,\textrm{in}\,\,\,\Omega\times (T_1,T_2)\,\,\,\textrm{with}\,\,\,\Omega\subset\mathbb{R}^3\,\,\,\textrm{open},\\
&\textrm{div}\,v=\textrm{div}\,b=0\\
&\textrm{and}\,\,\,\|b\|_{L^{\infty}(T_1,T_2; BMO^{-1}(\Omega))}\leq M.
\end{split}
\end{align}
The authors utilize a reverse H\"{o}lder inequality for $\nabla v$ derived in \cite{sereginreverseholder}. Specifically, for $l\in (\tfrac{6}{5},2)$ and the parabolic cube\footnote{$Q(z_0,R)$ is defined in `Preliminaries', see \eqref{paraboliccylinderdef}.} $Q(z_0,2\rho)\subset \Omega\times (T_1,T_2)$: 
\begin{align}\label{reverseholder}
\begin{split}
\frac{1}{|Q(\rho)|}\int\limits_{Q(z_0,\rho)} |\nabla v|^2 dxds\leq C(M,l)\Big(\frac{1}{|Q(2\rho)|}\int\limits_{Q(z_0,2\rho)} |\nabla v|^l dxds\Big)^{\frac{2}{l}}+C(M,l)\frac{1}{|Q(2\rho)|}\int\limits_{Q(z_0,2\rho)} |p|^2 dxds.
\end{split}
\end{align}
Using this and assuming that the pressure $p$ is square integrable, the authors in \cite{burczakseregin} show that \eqref{reverseholder} implies that the \textit{local maximal function} of $|\nabla v|^2$ is integrable. By a classical result of Stein \cite{stein}, this then implies that $|\nabla v|^2$ is locally in $L\log L$.

Traditionally, reverse H\"{o}lder inequalities for PDEs are used to provide higher integrability via \textit{Gehring's Lemma} (see \cite{gehring} and especially a modified version derived by Giaquinta and Modica in \cite{giaquintamodica}). To apply Gehring's lemma \cite{giaquintamodica} to the reverse H\"{o}lder inequality \eqref{reverseholder} requires the pressure $p$ to have space-time integrability greater than $2$. However, even for the Navier-Stokes equations ($b=v$) satisfying $\|v\|_{L^{\infty}_{t}L^{3,\infty}_{x}}\leq M$, before this paper it was only known apriori that the pressure is square integrable.

For the Navier-Stokes equations, one might hope it is possible to obtain a different version of \eqref{reverseholder} with a lower power of the pressure, whilst still being amenable to Gehring's lemma. Such reverse H\"{o}lder inequalities must scale correctly with respect to the Navier-Stokes scaling symmetry \eqref{eqrescaling}. In fact, any reverse H\"{o}lder inequality of the form \eqref{reverseholder} 
 satisfying both these constraints must involve the square integral of the pressure. Moreover a similar issue occurs if one seeks a reverse H\"{o}lder inequality with the pressure replaced by a power of $|v|$. This represents an obstacle for using Gehring's lemma to obtain higher integrability of $|\nabla v|^2$, under the assumption that $\|v\|_{L^{\infty}_{t}L^{3,\infty}_{x}}\leq M$.

Our strategy to prove the higher integrability of $v$ in Theorem \ref{higherintegpres} is elementary and is not based on Gehring's Lemma nor reverse H\"{o}lder inequalities. Instead it is based on taking a `\textit{small exponent}', which we outline below.
Following \cite{galdirionero} and \cite{beiraodevega}, for $q\in (2,3]$ we perform an energy estimate on $|v|^{\frac{q}{2}}$ on $\mathbb{R}^3\times (t_1,T^*)$. Here, $T^*$ is the first blow-up time. This gives that for $t\in (t_1,T^*)$:
\begin{equation}
\label{e.enerbal2intro}
\int\limits_{\R^3}|v(x,t)|^q\, dx+\frac{4(q-1)}{q}\int\limits_{t_1}^t\int\limits_{\R^3}|\nabla|v|^{\frac q2}|^2\, dxds\leq\int\limits_{\R^3}|v(x,t_1)|^q\, dx+{2(q-2)}\int\limits_{t_1}^t\int\limits_{\R^3}|p||v|^{\frac{q-2}{2}}|\nabla(|v|^{\frac{q}{2}})|\, dxds.
\end{equation}
We then estimate the last term of \eqref{e.enerbal2intro} as in \cite{berselligaldi} and \cite{caifangzhai} to obtain
\begin{align}\label{e.energybal3intro}
\begin{split}
\int\limits_{\R^3}|v(x,t)|^q\, dx+\frac{4(q-1)}{q}\int\limits_{t_1}^t\int\limits_{\R^3}|\nabla|v|^{\frac q2}|^2\, dxds&\leq\int\limits_{\R^3}|v(x,t_1)|^q\, dx\\
&+C_{3,univ}{(q-2)}\|p\|^{\frac{1}{2}}_{L^{\infty}(0,T^*; L^{\frac{3}{2},\infty}(\mathbb{R}^3))}\int\limits_{t_1}^t\int\limits_{\R^3}|\nabla(|v|^{\frac{q}{2}}) |^2\, dxds.
\end{split}
\end{align}
If $q=3$ and $\|p\|_{L^{\infty}(0,T^*; L^{\frac{3}{2},\infty}(\mathbb{R}^3))}$ is small enough, we can absorb the last term in the right-hand-side of \eqref{e.energybal3intro} into the left-hand-side of \eqref{e.energybal3intro} and one can then infer that $v$ does not blow-up at $T^*$. This was done in \cite{caifangzhai}. If $q=3$ and $\|p\|_{L^{\infty}(0,T^*; L^{\frac{3}{2},\infty}(\mathbb{R}^3))}\leq M^2$ is large, then this does not work and regularity under this assumption remains unknown. In this case, we instead compensate for $M$ being large by choosing an exponent $q\sim 2+O(\tfrac{1}{M})$ such that $q-2$ is small enough for the absorption to take place. This gives higher integrability up to the first blow-up time.

Let us mention that exploiting a `\textit{small exponent}' has also been used for other equations of hydrodynamical origin. Notably it was used to give an alternative proof of the regularity of the critically dissipative SQG equation in \cite{constantinvicoltarfuleaSQG}. On the other side of the coin, `small exponents' have also played a role in \cite{elgindi} for constructing singular $C^{1,\alpha}$ solutions to the three-dimensional Euler equation, with $\alpha>0$ being small.
\subsubsection{Theorem \ref{finiteTypeItimeslicesing}} 
The investigation of the structure of the singular set for the Navier-Stokes equations has a long history. Leray's seminal work \cite{Le} implied that the set of singular times for a weak Leray solution has Hausdorff dimension at most $\tfrac{1}{2}$. The investigation of the space-time singular set $\mathcal{S}$ was initiated by Scheffer in \cite{scheffer}. Specifically, in \cite{scheffer} it was shown that $\mathcal{S}$ has Hausdorff dimension at most 2. In the subsequent seminal work \cite{CKN}, it was shown that the parabolic one-dimensional Hausdorff dimension of the singular set $\mathcal{S}$ is zero. Since \cite{CKN}, there have been numerous investigations into the size of the singular set. We (non-exhaustively) list \cite{giga}, \cite{choelewis}, \cite{robinsonsadowski}, \cite{neustupa}, \cite{sereginfiniteL3} and \cite{wangzhang}.

Recently in \cite{choewolfyang}, the authors consider a suitable finite energy solution $v:\mathbb{R}^3\times [0,T^*]\rightarrow\mathbb{R}^3$, which loses smoothness at $T^*$ and satisfies the Type I bound
\begin{equation}\label{TypeIintro}
\|v\|_{L^{\infty}(0,T^*; L^{3,\infty}(\mathbb{R}^3))}\leq M.
\end{equation}
Specifically, it was shown in \cite{choewolfyang} that the scale-invariant bound \eqref{TypeIintro} implies that $v$ has a finite number of singular points at $T^*$. Subsequently this result was extended in \cite{sereginnote}, where it was shown that the same result holds for \textit{suitable weak solutions}\footnote{The class of suitable weak solutions are defined in `2. Preliminaries'.} and up to the flat boundary (with $v$ satisfying a Dirichlet condition on the flat part of the boundary). Whilst these results are qualitative, a quantitative version was recently established by the author and Prange in \cite{barkerprange2020}. In \cite{barkerprange2020} it is shown that if $v$ satisfies \eqref{TypeIintro} then $v$ has at most $\exp(\exp(M^{1024}))$ singular points at $T^*$.

The results mentioned in the previous paragraph reply upon unique continuation and backward uniqueness results for parabolic operators proven in \cite{ESS}. Furthermore, the quantitative version in \cite{barkerprange2020} hinges on quantitative Carleman inequalities proven by Tao \cite{tao}. The application of such technology to the Navier-Stokes equations uses some specific structure of the Navier-Stokes equations, in particular the vorticity equation.

In Theorem \ref{finiteTypeItimeslicesing}, we show that $v$ has a finite number of singular points at a first blow-up time under a much weaker assumption than \eqref{TypeIintro}. We only assume that $v$ is uniformily bounded on certain timeslices of $[-1,0]$ instead of the whole time interval. Moreover, the quantitative bound we obtain for the number of singular bounds is substantially better than in \cite{barkerprange2020} (polynomial in $M$, rather than double exponential). Our strategy to prove Theorem \ref{finiteTypeItimeslicesing} is elementary and substantially simpler than the aforementioned results. In particular, we do not utilize backward uniqueness results, unique continuation results nor quantitative Carleman  inequalities. Our proof is based on three basic ingredients:
\begin{itemize} 
\item[](i) Caffarelli-Kohn-Nirenberg's $\varepsilon$-regularity criterion \cite{CKN}
\medskip
\item[](ii) Global estimates for solutions of the Navier-Stokes equations, with initial data $\|v_0\|_{L^{3,\infty}(\mathbb{R}^3)}\leq M$ \cite{barkersersve}
\medskip
\item[] (iii) A rescaling argument, exploiting that $L^{3,\infty}(\mathbb{R}^3)$ is invariant with respect to the Navier-Stokes rescaling for the initial data. 
\end{itemize} 
 Previous approaches work solely with the $L^{\infty}_{t}L^{3,\infty}_{x}$ norm of $v$, which causes technical complications since the $L^{3,\infty}$ norm is not countably additive\footnote{In particular the $L^{3,\infty}$ quasi-norm of a function $f:\mathbb{R}^3\rightarrow \mathbb{R}$ can have a comparable presence over a countable number of disjoint scales. For example, for $f(x)=|x|^{-1}$ we readily see that for every $k\in\mathbb{Z}:$ 
 $$\|f\|_{L^{3,\infty}(2^k<|x|<2^{k+1})}\geq  (\tfrac{7\pi}{6})^{\frac{1}{3}}. $$  }. By contrast, in Theorem \ref{finiteTypeItimeslicesing} we use ingredient (ii) above to transfer the background assumption on $v$ \eqref{typeItimeslice} into global bounds on $v$ which are countably additive. By means of ingredients (iii) and (i), having such a control on global countably additive quantities of $v$ allows us to effectively bound the finite number of singular points.
 \begin{section}{Preliminaries}
\subsection{General Notation}

Throughout this paper we adopt the Einstein summation convention. For arbitrary vectors $a=(a_{i}),\,b=(b_{i})$ in $\mathbb{R}^{n}$ and for arbitrary matrices $F=(F_{ij}),\,G=(G_{ij})$ in $\mathbb{M}^{n}$ we put
 $$a\cdot b=a_{i}b_{i},\,|a|=\sqrt{a\cdot a},$$
 $$a\otimes b=(a_{i}b_{j})\in \mathbb{M}^{n},$$
 $$FG=(F_{ik}G_{kj})\in \mathbb{M}^{n}\!,\,\,F^{T}=(F_{ji})\in \mathbb{M}^{n}\!,$$
 $$F:G=
 F_{ij}G_{ij}\,\,\,\textrm{and}
 \,\,\,|F|=\sqrt{F:F}.$$
 For $x_0\in\mathbb{R}^n$ and $R>0$, we define the ball
\begin{equation}\label{balldef}
B(x_0,R):=\{x: |x-x_0|<R\}.
\end{equation}
For $z_0=(x_0,t_0)\in \mathbb{R}^n\times\mathbb{R}$, we denote the parabolic cylinder by
\begin{equation}\label{paraboliccylinderdef}
Q(z_0,R):=\{(x,t): |x-x_0|< R,\,t\in (t_0-R^2,t_0)\}.
\end{equation}
 Let $e^{t\Delta}u_{0}$ denote the heat kernel convoluted with $u_{0}$.

  If $X$ is a Banach space with norm $\|\cdot\|_{X}$, then $L_{s}(a,b;X)$, with $a<b$ and $s\in[1,\infty)$,  will denote the usual Banach space of strongly measurable $X$-valued functions $f(t)$ on $(a,b)$ such that
$$\|f\|_{L^{s}(a,b;X)}:=\left(\int\limits_{a}^{b}\|f(t)\|_{X}^{s}dt\right)^{\frac{1}{s}}<+\infty.$$ 
The usual modification is made if $s=\infty$.
Sometimes we will denote $L^{p}(0,T; L^{q})$ by $L^{p}_{T}L^{q}$ or $L^{p}(0,T; L^{q}_{x})$.
  
Let $C([a,b]; X)$ denote the space of continuous $X$ valued functions on $[a,b]$ with usual norm. In addition, let $C_{w}([a,b]; X)$ denote the space of $X$ valued functions, which are continuous from $[a,b]$ to the weak topology of $X$.

\subsection{Lorentz spaces}
Given a measurable subset $\Omega\subseteq\mathbb{R}^{n}$, let us define the Lorentz spaces. 
For a measurable function $f:\Omega\rightarrow\mathbb{R}$ define:
\begin{equation}\label{defdistchapter2}
d_{f,\Omega}(\alpha):=\mu(\{x\in \Omega : |f(x)|>\alpha\}),
\end{equation}
where $\mu$ denotes the Lebesgue measure on $\mathbb{R}^n$.
 The Lorentz space $L^{p,q}(\Omega)$, with $p\in [1,\infty)$, $q\in [1,\infty]$, is the set of all measurable functions $g$ on $\Omega$ such that the quasinorm $\|g\|_{L^{p,q}(\Omega)}$ is finite. Here:

\begin{equation}\label{Lorentznormchapter2}
\|g\|_{L^{p,q}(\Omega)}:= \Big(p\int\limits_{0}^{\infty}\alpha^{q}d_{g,\Omega}(\alpha)^{\frac{q}{p}}\frac{d\alpha}{\alpha}\Big)^{\frac{1}{q}},
\end{equation}
\begin{equation}\label{Lorentznorminftychapter2}
\|g\|_{L^{p,\infty}(\Omega)}:= \sup_{\alpha>0}\alpha d_{g,\Omega}(\alpha)^{\frac{1}{p}}.
\end{equation}\\
Notice that if $p\in(1,\infty)$ and $q\in [1,\infty]$ there exists a norm, which is equivalent to the quasinorm defined above, for which $L^{p,q}(\Omega)$ is a Banach space. 
For $p\in [1,\infty)$ and $1\leq q_{1}< q_{2}\leq \infty$, we have the following continuous embeddings 
\begin{equation}\label{Lorentzcontinuousembeddingchapter2}
L^{p,q_1}(\Omega) \hookrightarrow  L^{p,q_2}(\Omega)
\end{equation}
and the inclusion is known to be strict.

Let us recall a known proposition known as `O'Neil's convolution inequality' (Theorem 2.6 of  O'Neil's paper \cite{O'Neil}).
\begin{pro}\label{O'Neilchapter2}
Suppose $1< p_{1}, p_{2}, r<\infty$ and $1\leq q_{1}, q_{2}, s\leq\infty $ are such that
\begin{equation}\label{O'Neilindices1chapter2}
\frac{1}{r}+1=\frac{1}{p_1}+\frac{1}{p_{2}}
\end{equation}
and
\begin{equation}\label{O'Neilindices2chapter2}
\frac{1}{q_1}+\frac{1}{q_{2}}\geq \frac{1}{s}.
\end{equation}
Suppose that
\begin{equation}\label{fghypothesischapter2}
f\in L^{p_1,q_1}(\mathbb{R}^{d})\,\,\textrm{and}\,\,g\in  L^{p_2,q_2}(\mathbb{R}^{d}).
\end{equation}
Then
\begin{equation}\label{fstargconclusion1chapter2}
f\ast g \in L^{r,s}(\mathbb{R}^d)\,\,\rm{with} 
\end{equation}
\begin{equation}\label{fstargconclusion2chapter2}
\|f\ast g \|_{L^{r,s}(\mathbb{R}^d)}\leq 3r \|f\|_{L^{p_1,q_1}(\mathbb{R}^d)} \|g\|_{L^{p_2,q_2}(\mathbb{R}^d)}. 
\end{equation}
\end{pro}

We will make extensive use of the following known interpolative inequality for Lorentz spaces. For a proof, see (for example) Proposition 1.1.14 of \cite{grafakos}.
\begin{pro}\label{interpolativepropertyLorentz}
Suppose that $1\leq p<r<q\leq\infty$ and 
let $0<\theta<1$ be such that
$$\frac{1}{r}=\frac{\theta}{p}+\frac{1-\theta}{q}.$$
Then the assumption that $f\in L^{p,\infty}(\Omega)\cap L^{q,\infty}(\Omega)$ implies $f\in L^{r}(\Omega)$ with the estimate 
\begin{equation}\label{interpolationestimate}
\|f\|_{L^{r}(\Omega)}\leq \Big(\frac{r}{r-p}+\frac{r}{q-r}\Big)^{\frac{1}{r}}\|f\|_{L^{p,\infty}(\Omega)}^{\theta}\|f\|_{L^{q,\infty}(\Omega)}^{1-\theta}.
\end{equation}
\end{pro} 
\begin{subsection}{Solution classes of the Navier-Stokes equations}

 We say $v$ is a \emph{finite-energy solution} or a \emph{Leray-Hopf solution} to the Navier-Stokes equations on $(0,T)$ if $v\in C_{w}([0,T]; L^{2}_{\sigma}(\mathbb{R}^3))\cap L^{2}(0,T; \dot{H}^{1}(\mathbb{R}^3))$ and if it satisfies the global energy inequality 
$$\|v(\cdot,t)\|_{L^{2}(\mathbb{R}^3)}^2+2\int\limits_{0}^{t}\int\limits_{\mathbb{R}^3}|\nabla v|^2dxds\leq \|v(\cdot,0)\|_{L^{2}(\mathbb{R}^3)}^2.$$

Let $\Omega\subseteq\mathbb{R}^3$. We say that $(v,p)$ is a \textit{suitable weak solution} to the Navier-Stokes equations  
in $\Omega\times (T_{1},T)$ if it fulfills the properties described in \cite{gregory2014lecture} (Definition 6.1 p.133 in \cite{gregory2014lecture}).
\end{subsection}
\end{section}
\section{Higher integrability up to the first blow-up time}
\subsection{Proof of Theorem \ref{higherintegpres}}
We fix a constant $q\in (2,3)$ to be determined and take any $t_{1}\in (0,T^*)$
Note that the fact that $v$ is a weak Leray-Hopf solution satisfying \eqref{vbounded} implies that $v$ is sufficiently smooth on $\mathbb{R}^3\times (0,T^*)$ and satisfies
$$ v(\cdot,t_{1})\in L^{q}(\mathbb{R}^3).$$
Hence all calculations performed below can be rigorously justified.

Following \cite{beiraodevega} and \cite{galdirionero}, we test the Navier-Stokes system with $v|v|^{q-2}$. This yields the following $L^q$ energy balance
\begin{align}
\label{e.enerbal}
\begin{split}
&\frac{1}{q}\int\limits_{\R^3}|v(x,t)|^q\, dx+\int\limits_{t_1}^t\int\limits_{\R^3}|\nabla v|^2|v|^{q-2}\, dxds+\frac{4(q-2)}{q^2}\int\limits_{t_1}^t\int\limits_{\R^3}|\nabla|v|^{\frac q2}|^2\, dxds\\
&=\frac{1}{q}\int\limits_{\R^3}|v(x,t_1)|^q\, dx-\int\limits_{t_1}^t\int\limits_{\R^3}\nabla p\cdot|v|^{q-2}v\, dxds\\
&=\frac{1}{q}\int\limits_{\R^3}|v(x,t_1)|^q\, dx+I_{press}
\end{split}
\end{align}
for all $t\in[t_1,T^*)$. 
Here, 
\begin{equation}\label{Ipresfirstest}
I_{press}=\frac{2(q-2)}{q}\int\limits_{t_1}^t\int\limits_{\{x:\mathbb{R}^3: |v(x,s)|>0\}}p\frac{v_{k}}{|v|}|v|^{\frac{q-2}{2}}\partial_{k}(|v|^{\frac{q}{2}})\, dxds\leq \frac{2(q-2)}{q}\int\limits_{t_1}^t\int\limits_{\R^3}|p||v|^{\frac{q-2}{2}}|\nabla(|v|^{\frac{q}{2}})|\, dxds.
\end{equation}
Next, we see that
$$|\nabla(|v|^{\frac{q}{2}})|^2\leq \frac{q^2}{4}|\nabla v|^2|v|^{q-2}. $$ 
Combining this with \eqref{e.enerbal}-\eqref{Ipresfirstest} yields
\begin{align}
\label{e.enerbal2}
\begin{split}
&\int\limits_{\R^3}|v(x,t)|^q\, dx+\frac{4(q-1)}{q}\int\limits_{t_1}^t\int\limits_{\R^3}|\nabla|v|^{\frac q2}|^2\, dxds\\
&\leq\int\limits_{\R^3}|v(x,t_1)|^q\, dx+{2(q-2)}\int\limits_{t_1}^t\int\limits_{\R^3}|p||v|^{\frac{q-2}{2}}|\nabla(|v|^{\frac{q}{2}})|\, dxds
\end{split}
\end{align}
for all $t\in[t_1,T^*)$.

Now, we focus on estimating the last term in \eqref{e.enerbal2} under the assumption \eqref{pressuretypeI} on the pressure $p$. We do so using similar arguments as in \cite{berselligaldi}.

By the Sobolev embedding inequality, there exists a universal constant $C_{1, univ}\geq 1$ such that for $s\in (0,T^*)$:
$$\|v(\cdot,s)\|_{L^{3q}(\mathbb{R}^3)}^{\frac{q}{2}}=\||v|^{\frac{q}{2}}(\cdot,s)\|_{L^{6}(\mathbb{R}^3)}\leq C_{1,univ}\|\nabla(|v|^{\frac{q}{2}})(\cdot,s)\|_{L^{2}(\mathbb{R}^3)} .$$
Using that $q\in (2,3)$ we get
\begin{equation}\label{sobembeddingvq}
\|v(\cdot,s)\|_{L^{3q}(\mathbb{R}^3)}\leq C_{1,univ}\|\nabla(|v|^{\frac{q}{2}})(\cdot,s)\|_{L^{2}(\mathbb{R}^3)}^{\frac{2}{q}}.
\end{equation}
Furthermore, for $s\in (0,T^*)$
\begin{equation}\label{somembeddingvq-2}
\||v|^{\frac{q-2}{2}}(\cdot,s)\|_{L^{\frac{6q}{q-2}}(\mathbb{R}^3)}=\|v(\cdot,s)\|_{L^{3q}(\mathbb{R}^3)}^{\frac{q-2}{2}}\leq C_{1,univ}^{\frac{q-2}{2}}\|\nabla(|v|^{\frac{q}{2}})(\cdot,s)\|_{L^{2}(\mathbb{R}^3)}^{1-\frac{2}{q}}\leq C_{1,univ}\|\nabla(|v|^{\frac{q}{2}})(\cdot,s)\|_{L^{2}(\mathbb{R}^3)}^{1-\frac{2}{q}}.
\end{equation}
Applying \eqref{sobembeddingvq}-\eqref{somembeddingvq-2} and H\"{o}lder's inequality to 
\begin{equation}\label{presterm}
\int\limits_{t_1}^t\int\limits_{\R^3}|p||v|^{\frac{q-2}{2}}|\nabla(|v|^{\frac{q}{2}})|\, dxds
 \end{equation}
 gives
 \begin{align}\label{prestermholder}
 \begin{split}
 \int\limits_{t_1}^t\int\limits_{\R^3}|p||v|^{\frac{q-2}{2}}|\nabla(|v|^{\frac{q}{2}})|\, dxds&\leq \int\limits_{t_1}^t \|p(\cdot,s)\|_{L^{\frac{3q}{q+1}}(\mathbb{R}^3)}\||v|^{\frac{q-2}{2}}(\cdot,s)\|_{L^{\frac{6q}{q-2}}(\mathbb{R}^3)}\|\nabla(|v|^{\frac{q}{2}})(\cdot,s)\|_{L^{2}(\mathbb{R}^3)}ds\\
 &\leq C_{1,univ}\int\limits_{t_1}^t \|p(\cdot,s)\|_{L^{\frac{3q}{q+1}}(\mathbb{R}^3)}\|\nabla(|v|^{\frac{q}{2}})(\cdot,s)\|_{L^{2}(\mathbb{R}^3)}^{2-\frac{2}{q}} ds.
 \end{split}
 \end{align}
 Now we apply Proposition \ref{interpolativepropertyLorentz} to $\|p(\cdot,s)\|_{L^{\frac{3q}{q+1}}(\mathbb{R}^3)}$. Together with the assumption \eqref{pressuretypeI} and $q\in (2,3)$, this gives that for $s\in (0,T^*)$:
 \begin{equation}\label{presestinterpolation}
 \|p(\cdot,s)\|_{L^{\frac{3q}{q+1}}(\mathbb{R}^3)}\leq\Big(\frac{2q+2}{q-1}\Big)^{\frac{q+1}{3q}}\|p(\cdot,s)\|_{L^{\frac{3}{2},\infty}(\mathbb{R}^3)}^{\frac{1}{2}}\|p(\cdot,s)\|_{L^{\frac{3q}{2}}(\mathbb{R}^3)}^{\frac{1}{2}}\leq  6^{\frac{1}{2}}M\|p(\cdot,s)\|_{L^{\frac{3q}{2}}(\mathbb{R}^3)}^{\frac{1}{2}}.
 \end{equation}
 Next,  note that the associated pressure has the form $p=\mathcal{R}_{i}\mathcal{R}_{j}(v_{i}v_{j})$.
Here, $R=(R_{\alpha})_{\alpha=1,\ldots 3}$ is the Riesz transform and we utilize the Einstein summation convention.
Using this, Calder\'{o}n-Zygmund estimates, \eqref{sobembeddingvq} and $q\in (2,3)$, we see that there exists a universal constant $C_{2,univ}$ (independent of $q$) such that
$$\|p(\cdot,s)\|_{L^{\frac{3q}{2}}(\mathbb{R}^3)}^{\frac{1}{2}}\leq C_{2,univ}\|\nabla(|v|^{\frac{q}{2}})(\cdot,s)\|_{L^{2}(\mathbb{R}^3)}^{\frac{2}{q}}. $$
Combining this with \eqref{prestermholder}-\eqref{presestinterpolation} gives that there exists a positive universal constant $C_{3,univ}$ such that
\begin{equation}\label{pressureintegmainest}
2(q-2)\int\limits_{t_1}^t\int\limits_{\R^3}|p||v|^{\frac{q-2}{2}}|\nabla(|v|^{\frac{q}{2}})|\, dxds\leq C_{3,univ}M(q-2) \int\limits_{t_1}^t\int\limits_{\R^3}|\nabla|v|^{\frac q2}|^2\, dxds.
\end{equation}
Utilizing this and \eqref{e.enerbal2} now gives that for all $t\in (t_1,T)$:
\begin{align}\label{energybalpenultimate}
\begin{split}
&\int\limits_{\R^3}|v(x,t)|^q\, dx+\frac{4(q-1)}{q}\int\limits_{t_1}^t\int\limits_{\R^3}|\nabla|v|^{\frac q2}|^2\, dxds\\
&\leq\int\limits_{\R^3}|v(x,t_1)|^q\, dx+C_{3,univ}M{(q-2)}\int\limits_{t_1}^t\int\limits_{\R^3}|\nabla|v|^{\frac q2}|^2\, dxds.
\end{split} 
\end{align}
Define $C^{(0)}_{univ}:=\frac{1}{C_{3,univ}}$ and fix 
\begin{equation}\label{qdef}
q:=2+\frac{C^{(0)}_{univ}}{M}.
\end{equation} 
Clearly we 
have
$$C_{3,univ}M(q-2)\leq \frac{4(q-1)}{2q}. $$
Thus, taking $q$ as defined in \eqref{qdef}, we see that \eqref{energybalpenultimate} gives that for all $t\in (t_1,T)$
\begin{align}\label{energybalfinal}
\begin{split}
&\int\limits_{\R^3}|v(x,t)|^q\, dx+\frac{4(q-1)}{2q}\int\limits_{t_1}^t\int\limits_{\R^3}|\nabla|v|^{\frac q2}|^2\, dxds\\
&\leq\int\limits_{\R^3}|v(x,t_1)|^q\, dx.
\end{split}
\end{align}
From this we readily obtain the desired conclusion \eqref{higherintegrable}. 

\subsection{Proof of Corollary \ref{higherintegtypeI}}
Let us take any $t_{2}\in (0,T^*)$. We then take $t_{1}\in (0,t_{2})$.
The associated pressure to $v$ has the form $p=\mathcal{R}_{i}\mathcal{R}_{j}(v_{i}v_{j})$.
Here, $R=(R_{\alpha})_{\alpha=1,\ldots 3}$ is the Riesz transform and we utilize the Einstein summation convention.
Using this, \eqref{typeIhigherinteg}  and known Calder\'{o}n-Zygmund estimates for singular integrals of convolution type, we infer the following. Namely there exists a positive universal constant $C_{4,univ}$ such that
$$\|p\|_{L^{\infty}((0,T^*);L^{\frac{3}{2},\infty}(\mathbb{R}^3))}\leq C_{4,univ}M^2. $$
This allows us to apply Theorem \ref{higherintegpres}. In particular, there exists a positive universal constant $C_{5,univ}$ such that for 
\begin{equation}\label{qdefrecall}
q=2+\frac{C_{5,univ}}{M} 
\end{equation}
the following holds true. Namely,
\begin{equation}\label{higherintegrablerecall}
\||v|^{\frac{q}{2}}\|_{L^{\infty}((t_1,T^*); L^{2}(\mathbb{R}^3))}+\int\limits_{t_1}^{T^*}\int\limits_{\mathbb{R}^3} |\nabla (|v|^{\frac{q}{2}})|^2 dxds<\infty.
\end{equation}
By the Sobolev embedding theorem, we infer that
$$|v|^{\frac{q}{2}}\in L^{2}((t_{1},T^*); L^{6}(\mathbb{R}^3)). $$
This implies that
\begin{equation}\label{higherintegvsobolev}
v\in L^{2}((t_{1},T^*); L^{6+\frac{3C_{5,univ}}{M}}(\mathbb{R}^3)).
\end{equation}
Define
\begin{equation}\label{mdef}
m:=4+\frac{3C_{5,univ}}{6M+3C_{5,univ}}
\end{equation}
and
\begin{equation}\label{thetadef}
\theta:=\frac{2}{m}.
\end{equation}
With these definitions, we see that
\begin{equation}\label{indicesinterpolation1}
\frac{1}{m}=\frac{\theta}{6+\frac{3C_{5,univ}}{M}}+\frac{1-\theta}{3}.
\end{equation}
Applying Proposition \ref{interpolativepropertyLorentz}, we see that for $t\in (t_1,T^*)$ and $M$ sufficiently large:
$$\|v(\cdot,t)\|_{L^{m}(\mathbb{R}^3)}\leq {10}^{\frac{1}{4}}\|v(\cdot,t)\|_{L^{3,\infty}(\mathbb{R}^3)}^{1-\theta}\|v(\cdot,t)\|^{\theta}_{L^{6+\frac{3C_{5,univ}}{M}}(\mathbb{R}^3)}\leq 10^{\frac{1}{4}}M^{1-\theta} \|v(\cdot,t)\|^{\frac{2}{m}}_{L^{6+\frac{3C_{5,univ}}{M}}(\mathbb{R}^3)}.$$
This, together with \eqref{higherintegvsobolev} implies that
\begin{equation}\label{vL4plus}
|v|^2\in L^{2+\frac{3C_{5,univ}}{12M+6C_{5,univ}}}(\mathbb{R}^3\times (t_1,T^*)).
\end{equation}
For $t\in (t_1,\infty)$, it is known the $v$ can be represented in the following way:
\begin{align}\label{vduhamel}
\begin{split}
v(\cdot,t)&= e^{(t-t_1)\Delta}v(\cdot,t_1)+\int\limits_{0}^{t-t_1} e^{(t-t_1-s)\Delta}\mathbb{P}\partial_{i}(v_{i}(\cdot,s)v_{j}(\cdot,s)) ds\\
&=e^{(t-t_1)\Delta}v(\cdot,t_1)+\partial_{i}\Big(\int\limits_{0}^{t-t_1} e^{(t-t_1-s)\Delta}\mathbb{P}(v_{i}(\cdot,s)v_{j}(\cdot,s)) ds\Big).
\end{split} 
\end{align}
Here,  we adopt the Einstein summation convention and $\mathbb{P}$ is the Leray projector onto divergence-free vector fields. Using maximal regularity for the heat equation (see Theorem 10.12 of \cite{brezis2010}), together with \eqref{vL4plus}, we see that
\begin{equation}\label{gradientnonlinL4+}
\nabla\Big(\int\limits_{0}^{t-t_1} e^{(t-t_1-s)\Delta}\mathbb{P}\partial_{i}(v_{i}(\cdot,s)v_{j}(\cdot,s)) ds\Big)\in L^{2+\frac{3C_{5,univ}}{12M+6C_{5,univ}}}(\mathbb{R}^3\times (t_1,T^*)).
\end{equation}
Since $v$ is a weak Leray-Hopf solution, we have that $v(\cdot,t_{1})\in L^{2}(\mathbb{R}^3)$. Noting that $t_{2}\in(t_{1}, T^*)$, it is clear that
$$\nabla e^{(t-t_1)\Delta}v(\cdot,t_1)\in L^{2+\frac{3C_{5,univ}}{12M+6C_{5,univ}}}(\mathbb{R}^3\times (t_2,T^*)). $$
Combining this with \eqref{gradientnonlinL4+} gives the desired conclusion. 

\section{Quantified finite number of singular points } 
The purpose of this section is to prove Theorem \ref{finiteTypeItimeslicesing}. First we state a known $\varepsilon$-regularity criteria for the Navier-Stokes equations. 

\begin{pro}[\cite{CKN} and {\cite[Theorem 30.1]{LR02}}]\label{CKN}
There exists absolute constants $\ep_{0}^{*}>0$ and $C_{CKN}\in(0,\infty)$ such that if $(v,p)$ is a suitable weak solution to the Navier-Stokes equations on $Q(0,1)$ and for some $\ep_{0}\leq \ep_{0}^{*}$ 
\begin{equation}\label{CKNsmallness}
\int\limits_{Q(0,1)} |v|^3+|p|^{\frac{3}{2}} dxdt\leq \ep_{0}
\end{equation}
then  one concludes that
 $v\in L^{\infty}(Q(0,\tfrac{1}{2})).$
\end{pro}
Before proving Theorem \ref{finiteTypeItimeslicesing}, we also need to state and prove a lemma below.
\begin{lemma}\label{energyboundweakL3}
There exists a positive universal constants $C^{(3)}_{univ}-C^{(5)}_{univ}$ such that the following holds true for every $M\geq 1$.

Let $v:\mathbb{R}^3\times (0,1)\rightarrow\mathbb{R}^3$ be a weak Leray-Hopf solution to the Navier-Stokes equations, with initial data $v_{0}$ and associated pressure $p:\mathbb{R}^3\times (0,1)\rightarrow \mathbb{R}$.

Suppose that
\begin{equation}\label{vfirstblowupenergybound}
v\in L^{\infty}([0,1); L^{\infty}(\mathbb{R}^3))
\end{equation}
and 
\begin{equation}\label{initialdataweakL3}
\|v_{0}\|_{L^{3,\infty}(\mathbb{R}^3)}\leq M.
\end{equation}
Define
$$ u(\cdot,t):= v(\cdot,t)-e^{t\Delta}v_{0}.$$
Then the above assumptions imply that
\begin{equation}\label{perturbationenergyest}
\|u\|^{2}_{L^{\infty}(0,1; L^{2}(\mathbb{R}^3))}+\int\limits_{0}^{1}\int\limits_{\mathbb{R}^3} |\nabla u|^2 dxds\leq C^{(3)}_{univ}M^{12},
\end{equation}
\begin{equation}\label{L103}
\|v\|_{L^{\frac{10}{3}}(\mathbb{R}^3\times (0,1))}\leq C^{(4)}_{univ}M^{6}
\end{equation} 
and
\begin{equation}\label{presL53}
\|p\|_{L^{\frac{5}{3}}(\mathbb{R}^3\times (0,1))}\leq C^{(5)}_{univ}M^{12}
\end{equation}
\end{lemma}
\begin{proof}
Utilizing Lemma 3.4 of \cite{barkersersve} we see that for all $N>0$ we have
\begin{align}\label{perturbationenergyN}
\begin{split}
&\|u\|^{2}_{L^{\infty}(0,1; L^{2}(\mathbb{R}^3))}+\int\limits_{0}^{1}\int\limits_{\mathbb{R}^3} |\nabla u|^2 dxds\\&\leq C_{6,univ}(N^{-1}\|v_{0}\|_{L^{3,\infty}(\mathbb{R}^3)}^3+N^{\frac{2}{5}}\|v_{0}\|_{L^{3,\infty}(\mathbb{R}^3)}^{\frac{18}{5}})\\
&+C_{6,univ}\exp(C_{6,univ}N^{\frac{1}{2}}\|v_{0}\|_{L^{3,\infty}(\mathbb{R}^3)}^{\frac{9}{2}})(N^{-\frac{1}{2}}\|v_{0}\|_{L^{3,\infty}(\mathbb{R}^3)}^{\frac{33}{8}}+N^{\frac{9}{10}}\|v_{0}\|_{L^{3,\infty}(\mathbb{R}^3)}^{\frac{199}{40}})\\
&\leq C_{6,univ}(N^{-1}M^3+N^{\frac{2}{5}}M^{\frac{18}{5}})+C_{6,univ}\exp(C_{6,univ}N^{\frac{1}{2}}M^{\frac{9}{2}})(N^{-\frac{1}{2}}M^{\frac{33}{8}}+N^{\frac{9}{10}}M^{\frac{199}{40}}).
\end{split}
\end{align}
Choosing $N=M^{-9}$ then immediately gives \eqref{perturbationenergyest}.

Next we prove \eqref{L103}. Using Lebesgue interpolation and the Sobolev inequality, we have that for $t\in (0,1)$
$$\|u(\cdot,t)\|_{L^{\frac{10}{3}}(\mathbb{R}^3)}^{\frac{10}{3}}\leq \|u(\cdot,t)\|_{L^{6}(\mathbb{R}^3)}^{2}\|u(\cdot,t)\|_{L^{2}(\mathbb{R}^3)}^{\frac{4}{3}}\leq C_{7,univ}\|\nabla u(\cdot,t)\|_{L^{2}(\mathbb{R}^3)}^2 \|u(\cdot,t)\|_{L^{2}(\mathbb{R}^3)}^{\frac{4}{3}}.$$
Using \eqref{perturbationenergyest}, we have
\begin{equation}\label{perturbationL103}
\int\limits_{0}^{1}\int\limits_{\mathbb{R}^3} |u|^{\frac{10}{3}} dxds\leq C_{7,univ}\int\limits_{0}^{1}\int\limits_{\mathbb{R}^3} |\nabla u|^2 dxds \|u\|^{\frac{4}{3}}_{L^{\infty}(0,1; L^{2}(\mathbb{R}^3))}\leq C_{7,univ}M^{20}.
\end{equation}
Using \eqref{initialdataweakL3} and O'Neil's convolution inequality gives that for any $s\in (0,\infty):$
$$\|e^{s\Delta} v_{0}\|_{L^{\frac{10}{3}}(\mathbb{R}^3)}\leq \frac{C_{8,univ} M}{s^{\frac{1}{20}}}.  $$
Thus,
\begin{equation}\label{heatL103}
\int\limits_{0}^{1}\int\limits_{\mathbb{R}^3} |e^{s\Delta}u_{0}|^{\frac{10}{3}} dxds\leq C_{9,univ}M^{\frac{10}{3}}.
\end{equation}
Combining this with \eqref{perturbationL103} gives the desired conclusion \eqref{L103} for $M\geq 1$. 

Finally, since the associated pressure $p$ to $v$ is a Calder\'{o}n-Zygmund operator acting on $v\otimes v$, we see that \eqref{L103} implies \eqref{presL53}.
\end{proof}
\subsection{Proof of Theorem \ref{finiteTypeItimeslicesing}}
Let $\{x_{1},x_2\ldots x_{L}\}\subset \sigma$. Where $\sigma$ is defined as in \eqref{sigmadef}. In what follows, we will show that necessarily,
\begin{equation}\label{pupper}
L\leq C^{(2)}_{univ}M^{20}.
\end{equation}
This is sufficient to infer the conclusion of Theorem \ref{finiteTypeItimeslicesing}.

Regarding $\{x_{1}, x_{2},\ldots x_{L}\}\subset\sigma$, since $s_{n}\uparrow 0$, there exists $n=n(x_{1},\ldots x_{L})>0$ such that
\begin{equation}\label{separatesing}
\min_{\{(i,j)\in \{1,2,\ldots L\}^2: i\neq j\}} |x_{i}-x_{j}|\geq 2(-s_{n})^{\frac{1}{2}}. 
\end{equation}
Now we perform the Navier-Stokes rescaling
\begin{equation}\label{rescalefinitesingpoints} (\tilde{v}(x,t), \tilde{p}(x,t))= (\lambda v(\lambda x, \lambda ^2 t),\lambda^2 p(\lambda x, \lambda ^2 t))\,\,\,\textrm{with}\,\,\,\lambda:=(-s_{n})^{\frac{1}{2}}<1.
\end{equation}
Here, $\tilde{v}:\mathbb{R}^3\times (\frac{1}{s_{n}},\infty)\rightarrow\mathbb{R}^3$ and $p:\mathbb{R}^3\times (\frac{1}{s_{n}},\infty)\rightarrow\mathbb{R}$.

By \eqref{typeItimeslice}, we see that
\begin{equation}\label{rescaleminus1}
\|\tilde{v}(\cdot,-1)\|_{L^{3,\infty}(\mathbb{R}^3)}\leq M.
\end{equation} 
Hence, using Lemma \ref{energyboundweakL3} we see that
\begin{equation}\label{vtildeL103}
\int\limits_{-1}^{0}\int\limits_{\mathbb{R}^3} |\tilde{v}|^{\frac{10}{3}}+|\tilde{p}|^{\frac{5}{3}} dxds\leq C_{10,univ} M^{20}.
\end{equation}

Next define
\begin{equation}\label{yidef}
y_{i}:=\frac{x_{i}}{(-s_{n})^{\frac{1}{2}}}\,\,\,\textrm{for}\,\,\,i\in\{1,\ldots L\}.
\end{equation}
Then, 
\begin{equation}\label{yising}
\{(y_{i},0): i=1\ldots L\}\,\,\textrm{are}\,\,\textrm{singular}\,\,\textrm{points}\,\,\textrm{of}\,\, \tilde{v}.
\end{equation}
Using \eqref{separatesing}, we have that
\begin{equation}\label{disjointjballs}
B(y_{i}, 1)\cap B(y_{j}, 1)=\emptyset\,\,\,\forall i\neq j\in\{1,\ldots L\}.
\end{equation}
Utilizing Proposition \ref{CKN}, together with \eqref{yising}, we see that there exists a positive universal constant $\varepsilon^{*}_{0}$ such that
\begin{equation}\label{vtildeconversewolf}
\int\limits_{-1}^{0}\int\limits_{B(y_{i},1)} |\tilde{v}|^{\frac{10}{3}}+|\tilde{p}|^{\frac{5}{3}} dxds\geq\varepsilon^{*}_{0}\,\,\,\forall i\in\{1,\ldots,L\}.
\end{equation}
Summing this over $i=1\ldots L$, together with the use of \eqref{vtildeL103} and \eqref{disjointjballs}, gives
$$\varepsilon^{*}_{0}L\leq \sum_{i=1}^{L} \int\limits_{-1}^{0}\int\limits_{B(y_{i},1)} |\tilde{v}|^{\frac{10}{3}}+|\tilde{p}|^{\frac{5}{3}} dxds=\int\limits_{-1}^{0}\int\limits_{\cup_{i=1}^{L}B(y_{i},1)} |\tilde{v}|^{\frac{10}{3}} +|\tilde{p}|^{\frac{5}{3}} dxds\leq C_{10,univ} M^{20}. $$
Thus, $$L\leq \frac{C_{10,univ} M^{20}}{\varepsilon^{*}_{0}}$$ as required.  
\begin{remark}
The observant reader will notice that, under the assumptions of Theorem \ref{finiteTypeItimeslicesing}, $M^{20}$ is not an optimal bound for the number of singular points of $v$ at $t=0$. In the above proof of Theorem \ref{finiteTypeItimeslicesing}, it suffices to estimate global bounds for the $L^{s}_{x,t}$ space-time norm of $v$ for any $s\in (3,\frac{10}{3}]$. Using Lemma \ref{energyboundweakL3} to estimate this quantity then implies that the number of singular points at $t=0$ can be bounded by
\begin{equation}\label{improvedbound}
C_{s}M^{6s}\,\,\,\textrm{for}\,\,\textrm{any}\,\,s\in (3,\tfrac{10}{3}].
\end{equation}

If instead of (13), we assume the stronger assumption that
\begin{equation}\label{typeIfinal}
\|v\|_{L^{\infty}(-1,0; L^{3,\infty}(\mathbb{R}^3))}\leq M
\end{equation}
then the power of $M$ in \eqref{improvedbound} can be further improved. Following the energy estimates of \cite{tao} and \cite{barkerprange2020}, we can show that under the assumption \eqref{typeIfinal},  the number of singular points of $v$ at $t=0$ is at most
$$C_{s}M^{2s}\,\,\,\textrm{for}\,\,\textrm{any}\,\,s\in (3,\tfrac{10}{3}].  $$

\end{remark}

\bibliography{refs}        
\bibliographystyle{plain}

\end{document}